\documentclass[smallextended,numbook]{svjour3}     
\smartqed  
\usepackage{graphicx}
\usepackage{amsfonts}
\usepackage{amssymb}
\usepackage{amsmath}
\usepackage{amsxtra}
\usepackage{latexsym}
\begin{document}

\title{\bf \Large On the order optimality of the regularization via inexact Newton iterations
}

\titlerunning{Short form of title}        

\author{Qinian Jin        
}

\authorrunning{Short form of author list} 

\institute{Qinian Jin \at
Mathematical Sciences Institute,
The Australian National University,\\
Canberra, ACT 0200, Australia\\
\email{Qinian.Jin@anu.edu.au}
}


\newtheorem{Assumption}{Assumption}[section]

\def\d{\delta}
\def\ds{\displaystyle}
\def\e{{\epsilon}}
\def\eb{\bar{\eta}}
\def\enorm#1{\|#1\|_2}
\def\Fp{F^\prime}
\def\fishpack{{FISHPACK}}
\def\fortran{{FORTRAN}}
\def\gmres{{GMRES}}
\def\gmresm{{\rm GMRES($m$)}}
\def\Kc{{\cal K}}
\def\norm#1{\|#1\|}
\def\wb{{\bar w}}
\def\zb{{\bar z}}

\def\N{\mathcal N}
\def\R{\mathcal R}
\def\D{\mathcal D}
\def\X{\mathcal X}
\def\Y{\mathcal Y}
\def\B{\mathcal B}
\def\A{\mathcal A}
\def\a{\alpha}
\def\b{\beta}
\def\d{\delta}
\def\la{\lambda}

\maketitle

\begin{abstract}
Inexact Newton regularization methods have been proposed by Hanke and Rieder
for solving nonlinear ill-posed inverse problems. Every such a method consists of two components:
an outer Newton iteration and an inner scheme providing increments by regularizing local
linearized equations.  The method is terminated by a discrepancy principle. In this paper we
consider the inexact Newton regularization methods with the inner scheme defined by Landweber
iteration, the implicit iteration, the asymptotic regularization and Tikhonov regularization.
Under certain conditions we obtain the order optimal convergence rate result which
improves the suboptimal one of Rieder. We in fact obtain a more general order optimality
result by considering these inexact Newton methods in Hilbert scales.


\end{abstract}

\def\theequation{\thesection.\arabic{equation}}
\catcode`@=12

\section{\bf Introduction}
\setcounter{equation}{0}

Inverse problems arise whenever one searches for unknown causes based on observation of
their effects. Driven by the requirements from huge amount of practical applications,
the field of inverse problems has undergone a tremendous growth. Such problems are
usually ill-posed in the sense that their solutions  do not depend continuously on the data.
In practical applications, one never has exact data,
instead only noisy data are available due to errors in the measurements. Even if the deviation is
very small, algorithms developed for well-posed problems may fail, since noise could be amplified
by an arbitrarily large factor. Therefore, the development of stable methods for solving inverse
problems is a central topic.


In this paper we consider the stable resolution of nonlinear inverse problems
which mathematically can be formulated as the nonlinear equations
\begin{equation}\label{1}
F(x)=y,
\end{equation}
where $F: \D(F)\subset \X\mapsto \Y$ is a nonlinear Fr\'{e}chet differentiable operator
between two Hilbert spaces $\X$ and $\Y$ whose norms and inner
products are denoted as $\|\cdot\|$ and $(\cdot, \cdot)$
respectively. We use $F'(x)$ to denote the Fr\'{e}chet derivative of $F$ at $x\in \D(F)$
and use $F'(x)^*$ to denote the adjoint of $F'(x)$. We assume that (\ref{1}) has a solution $x^\dag$
in the domain $\D(F)$ of $F$, i.e. $F(x^\dag)=y$. Let $y^\d$ be the only available noisy data
of $y$ satisfying
\begin{equation}\label{1.2}
\|y^\delta-y\|\le \delta
\end{equation}
with a given small noise level $\delta> 0$. Due to the intrinsic ill-posedness, regularization
methods should be employed to produce from $y^\d$ a stable approximate solution of (\ref{1}).

Many regularization methods have been considered in the last two
decades. Due to their straightforward implementation and fast convergence property,
Newton type regularization methods are attractive for solving nonlinear inverse problems.
In \cite{Jin2011} we considered a general class of Newton type methods of the form
\begin{equation}\label{m1}
x_{n+1} =x_n+g_{t_n} \left(F'(x_n)^*F'(x_n)\right) F'(x_n)^*
\left(y^\d-F(x_n)\right),
\end{equation}
where $x_0$ is an initial guess of $x^\dag$, $\{t_n\}$ is a sequence of
positive numbers, and $\{g_t\}$ is a family of spectral filter functions.
The scheme (\ref{m1}) can be derived by applying the linear regularization
method defined by $\{g_t\}$ to the linearized equation
$$
F'(x_n)(x-x_n)=y^\delta-F(x_n)
$$
which follows from (\ref{1}) by replacing $y$ by $y^\d$ and $F(x)$ by
its linearization $F(x_n)+F'(x_n)(x-x_n)$ at $x_n$.
When the sequence $\{t_n\}$ is given a priori with suitable property, we showed in \cite{Jin2011}
that, under the discrepancy principle, the methods are convergent and order optimal.
We also considered in \cite{JT2011} the methods in Hilbert scales and obtained the order optimal convergence rates.

In the definition of the Newton type methods (\ref{m1}),
one may determine the sequence $\{t_n\}$ adaptively
during computation. Motivated by the inexact Newton methods
in \cite{DES82} for well-posed problems, Hanke proposed in \cite{H97} his regularizing
Levenberg-Marquardt scheme for solving nonlinear inverse problems with $\{t_n\}$ chosen to satisfy
$$
\|y^\d-F(x_n)-F'(x_n)(x_{n+1}-x_n)\|= \eta \|y^\d-F(x_n)\|
$$
at each step for some preassigned number $\eta\in (0,1)$ and with the discrepancy principle
used to terminate the iteration. Rieder generalized the idea in \cite{H97} and
proposed in \cite{R99} (see also \cite{LR2010}) a general
class of inexact Newton methods; every such a method consists of two components: an outer Newton
iteration and an inner scheme providing increment by regularizing
local linearized equations. When the inner scheme is defined by an iterative method, the number of iterations
is determined adaptively which has the advantage to avoid the over-solving of the
linearized equation that may occur when the inner scheme is terminated a priori.
The convergence rates of inexact Newton regularization methods were considered in \cite{R01} but only suboptimal
ones were derived. It is a longstanding question whether the inexact Newton methods are order optimal.
Important progress has been made recently in \cite{H2010} where the regularizing Levenberg-Marquardt
scheme is shown to be order optimal. In this paper we consider the inexact Newton regularization methods
in which the inner schemes are defined by applying various linear regularization methods, including Landweber iteration, the
implicit iteration, the asymptotic regularization and Tikhonov regularization, to the local linearized
equations and show that these methods are indeed order optimal by exploiting
ideas developed in \cite{H2010,JT2011,LR2010}.
We even consider these methods in Hilbert scales and derive the order optimal convergence rates.
Our theoretical results confirm those numerical illustrations in \cite{R99,R01}.

This paper is organized as follows. In Section 2 we formulate the methods precisely
and state the main results on the order optimal convergence rates. In Section 3 we show that these methods
are well-defined, and prove that the error decays monotonically. In Section 4 we complete the
proof of the the main result by deriving the order optimal convergence rates.

\section{\bf Main results}
\setcounter{equation}{0}

The inexact Newton regularization methods are a family of methods for
solving nonlinear ill-posed inverse problems. Every such a method consists of two components,
an outer Newton iteration and an inner scheme providing increments by regularizing
local linearized equations. An approximate solution is output by a discrepancy principle.

To be more precise, the method starts with an initial guess $x_0\in \D(F)$.
Assume that $x_n$ is a current iterate, one may apply any regularization scheme to
the linearized equation
\begin{equation}\label{linear}
F'(x_n) u=y^\d-F(x_n)
\end{equation}
to produce a family of regularized approximations $\{u_n(t)\}$.
One may choose $t_n$ to be the smallest number $t_n>0$ such that
\begin{equation} \label{4.2}
\|y^\d-F(x_n)-F'(x_n) u_n(t_n)\|\le \eta \|y^\d-F(x_n)\|
\end{equation}
for some preassigned value $0<\eta<1$. The next iterate is then updated as
$x_{n+1}=x_n +u_n(t_n)$. The outer Newton iteration is terminated by the discrepancy
principle
\begin{equation}\label{4.3}
\|y^\d-F(x_{n_\d})\|\le \tau\d <\|y^\d-F(x_n)\|, \qquad 0\le n<n_\d
\end{equation}
for some given number $\tau>1$. This outputs an integer $n_\d$ and hence $x_{n_\d}$ which is
used to approximate the exact solution $x^\dag$.

The convergence rates of the inexact Newton regularization methods have been considered in \cite{R99,R01}.
It has been shown that if
$$
x_0-x^\dag\in \R((F'(x^\dag)^* F'(x^\dag))^\mu)
$$
for some $0<\mu\le 1/2$,  then there is a number $0<\mu_0<\mu$ such that
$$
\|x_{n_\d}-x^\dag\|=O(\d^{2(\mu-\mu_0)/(1+2\mu)})
$$
which is only suboptimal. It is a long-standing question whether the inexact Newton
regularization methods are order optimal.  Important progress has been made recently in
\cite{H2010} where the regularizing Levenberg-Marquardt scheme is proved to be order optimal.

In this paper we will consider the inexact Newton regularization methods in which the inner
schemes are defined by applying Landweber iteration, the implicit iteration, the asymptotic
regularization, or  Tikhonov regularization to the linearized equation (\ref{linear})
and show that these methods are indeed order optimal. For these four methods, $u_n(t)$ are defined by
$$
u_n(t)=g_t\left(F'(x_n)^*F'(x_n)\right) F'(x_n)^* \left(y^\d-F(x_n)\right)
$$
with the spectral filter functions $\{g_t\}$ given by
\begin{equation}\label{5.7.2}
g_t(\la)=\sum_{j=0}^{[t]-1} (1-\lambda)^j, \,\,\,\, \sum_{j=1}^{[t]} (1+\lambda)^{-j},
\,\,\,\, \frac{1}{\la} \left(1-e^{-t\la}\right), \,\,\,\, \left(\frac{1}{t}+\la\right)^{-1}
\end{equation}
respectively, where $[t]$ denotes the largest integer not greater than $t$.

We need the following standard condition which is known as the Newton-Mysovskii condition (see \cite{DES98}).

\begin{Assumption}\label{A0}
(a) There exists $K_0\ge 0$ such that
\begin{equation*}
\|[F'(x)-F'(z)] h\|\le K_0\|x-z\| \|F'(z) h\|, \quad \forall h\in \X
\end{equation*}
for all $x, z \in B_\rho(x^\dag)\subset \D(F)$, where $B_\rho(x^\dag)$ denotes the ball of radius
$\rho>0$ with center at $x^\dag$.

(b) $F$ is properly scaled so that $\|F'(x) \|
\le \Theta<1$ for all $x\in B_\rho(x^\dag)$.
\end{Assumption}

The order optimality of these four inexact Newton regularization methods is contained in the following result.

\begin{theorem}\label{T1}
Let $F$ satisfy Assumption \ref{A0}, let $\tau>2$ and $0<\eta<1$
be such that $ \tau\eta>2$, and let $x_0\in B_\rho(x^\dag)$.
If $K_0\|x_0-x^\dag\|$ is sufficiently small, then the inexact Newton regularization methods
with the inner scheme defined by Landweber iteration, the implicit iteration, the asymptotic regularization,
or Tikhonov regularization are well-defined and terminate after $n_\d=O(1+|\log \d|)$ iterations.
If, in addition, $x_0-x^\dag =(F'(x^\dag)^*F'(x^\dag))^\mu \omega$ for some
$\omega\in \mathcal{N}(F'(x^\dag))^\perp\subset \X$ and $0<\mu\le 1/2$
and if $K_0\|\omega\|$ is sufficiently small, then there holds
$$
\|x_{n_\d}-x^\dag\| \le C \|\omega\|^{\frac{1}{1+2\mu}} \d^{\frac{2\mu}{1+2\mu}}
$$
for some constant $C$ independent of $\d$ and $\|\omega\|$.
\end{theorem}

We will not give the proof of Theorem \ref{T1} directly. Instead, we will prove a more
general result by considering  these four inexact Newton regularization methods in Hilbert scales.
Let $L$ be a densely defined self-adjoint strictly positive linear
operator in $\X$ satisfying
\begin{equation*}
\|x\|^2\le \gamma (L x, x), \quad x\in \D(L)
\end{equation*}
for some constant $\gamma>0$, where $\D(L)$ denotes the domain of $L$. For each $t\in {\mathbb R}$,
we define $\X_t$ to be the completion of $\cap_{k=0}^\infty \D(L^k)$
with respect to the Hilbert space norm
$$
\|x\|_t:= \|L^t x\|.
$$
This family of Hilbert spaces $\{\X_t\}_{t\in {\mathbb R}}$ is called the
Hilbert scales generated by $L$. The following are fundamental properties (see \cite{EHN96}):

(a) For any $-\infty<q<r<\infty$, $\X_r$ is densely and continuously embedded
into $\X_q$ with
\begin{equation}\label{embed}
\|x\|_q\le \gamma^{r-q}\|x\|_r, \quad x\in \X_r,
\end{equation}

(b) For any $-\infty<p<q<r<\infty$ there holds the interpolation inequality
\begin{equation}\label{inter}
\|x\|_q\le \|x\|_p^{\frac{r-q}{r-p}}\|x\|_r^{\frac{q-p}{r-p}},  \quad x\in \X_r.
\end{equation}

(c) If $T:\X\mapsto \Y$ is a bounded linear operator satisfying
$$
m\|h\|_{-a}\le \|Th\|\le M\|h\|_{-a}, \quad h\in \X
$$
for some constants $M\ge m>0$ and  $a\ge 0$, then for the operator
$A:=TL^{-s}:\X\mapsto \Y$ with $s\ge -a$ there holds for any $|\nu|\le 1$ that
\begin{equation}\label{2.3}
 \underline{c}(\nu) \|h\|_{-\nu(a+s)}\le\|(A^*A)^{\nu/2}h\| \le
\overline{c}(\nu) \|h\|_{-\nu(a+s)}
\end{equation}
on $\D((A^*A)^{\nu/2})$, where $A^*:=L^{-s} T^*: \Y\to \X$ is the adjoint of $A$ and
$$
\underline{c}(\nu):=\min\{m^\nu,M^\nu\} \quad \mbox{and} \quad
\overline{c}(\nu)=\max\{m^\nu, M^\nu\}.
$$

We will consider the inexact Newton regularization methods in which the inner schemes
are defined by applying Landweber iteration, the implicit iteration, the asymptotic regularization,
or Tikhonov regularization in Hilbert scales to the linearized equation (\ref{linear}). Now we have
\begin{equation}\label{5.7.1}
u_n(t)= g_{t}\left(L^{-2s}F'(x_n)^* F'(x_n)\right) L^{-2s} F'(x_n)^*\left(y^\d-F(x_n)\right)
\end{equation}
with $g_t$ defined by (\ref{5.7.2}), where $s\in {\mathbb R}$ is a suitable chosen number. The iterative solutions are defined by
$x_{n+1}=x_n+ u_n(t_n)$ with $t_n>0$ chosen to be the smallest number satisfying (\ref{4.2}).
The iteration is then terminated by the discrepancy principle (\ref{4.3}) to output an
approximate solution $x_{n_\d}$.

We will use $x_{n_\d}$, constructed from these four inexact Newton regularization methods in Hilbert scales,
to approximate the true solution $x^\dag$ of (\ref{1})
and derive the order optimal convergence rate when $x_0-x^\dag\in \X_\mu$ with $s<\mu\le b+2s$.
We need the following condition on the nonlinear operator $F$.

\begin{Assumption}\label{A1}
(a) There exist constants $a\ge 0$ and $0<m\le M<\infty$ such that
$$
m\|h\|_{-a}\le \|F'(x) h\|\le M \|h\|_{-a}, \quad h\in \X
$$
for all $x\in B_\rho(x^\dag)$.

(b) $F$ is properly scaled so that $\|F'(x) L^{-s}\|_{\X\to \Y}
\le \Theta<1$ for all $x\in B_\rho(x^\dag)$, where $s\ge -a$.

(c) There exist $0<\beta\le 1$, $0\le b\le a$ and $K_0\ge 0$ such that
\begin{equation*}
\|F'(x)-F'(z)\|_{\X_{-b}\to \Y}\le K_0\|x-z\|^\beta
\end{equation*}
for all $x, z\in B_\rho(x^\dag)$.

\end{Assumption}

This condition was first used in \cite{N2000} for the convergence
analysis of the nonlinear Landweber iteration in Hilbert scales. It was then used recently in \cite{HP08}
and \cite{JT2011} for nonlinear Tikhonov regularization and some Newton-type regularization methods
in Hilbert scales respectively. One can consult \cite{N2000,HP08} for several examples
satisfying Assumption \ref{A1}.

\begin{theorem}\label{T2}
Let $F$ satisfy Assumption \ref{A1} with $s\ge (a-b)/\beta$, let $\tau>2$ and $0<\eta<1$
be such that $ \tau\eta>2$, and let $x_0\in \D(F)$ be such that $\gamma^s\|x_0-x^\dag\|_s\le \rho$.
If $K_0\|x_0-x^\dag\|_s^\beta$ is sufficiently small, then the inexact Newton regularization
methods with the inner scheme defined by Landweber iteration, the implicit iteration, the asymptotic regularization,
or Tikhonov regularization in Hilbert scales are well-defined and terminate after $n_\d=O(1+|\log \d|)$ iterations.
If, in addition, $x_0-x^\dag \in \X_\mu$ for some $s<\mu \le b+2s$ and
$K_0\|x_0-x^\dag\|_\mu^\beta$ is sufficiently small, then there holds
$$
\|x_{n_\d}-x^\dag\|_r \le C \|x_0-x^\dag\|_\mu^{\frac{a+r}{a+\mu}} \d^{\frac{\mu-r}{a+\mu}}
$$
for all $r\in [-a, s]$, where $C$ is a constant independent of $\d$ and $\|x_0-x^\dag\|_\mu$.
\end{theorem}

The proof of Theorem \ref{T2} will be given in the next two sections. Here some remarks are in order.

\begin{remark}
When the inner scheme is defined by the asymptotic regularization or Tikhonov regularization, there is flexibility
to choose $t_n$ to satisfy
$$
\eta_1\|y^\d-F(x_n)\|\le \|y^\d-F(x_n)-F'(x_n) u_n(t_n)\|\le \eta_2\|y^\d-F(x_n)\|
$$
with some numbers $0<\eta_1\le \eta_2<1$. Furthermore, we only need $\tau>2$ and $\tau \eta_1>1$ in the convergence analysis.
\end{remark}

\begin{remark}
When $s>(a-b)/\beta$, the same order optimal convergence rate in Theorem \ref{T2} holds
for $x_0-x^\dag\in \X_\mu$ with $s\le \mu\le b+2s$ which can be seen from the proof of Lemma \ref{P1} in Section 4.
\end{remark}

\begin{remark}
If the Fr\'{e}chet derivative $F'(x)$ satisfies the Lipschitz condition
$$
\|F'(x)-F'(z)\|\le K_0\|x-z\|, \qquad x, z\in B_\rho(x^\dag),
$$
then Assumption \ref{A1} (c) holds with $b=0$ and
$\beta=1$, and thus, for these inexact Newton regularization methods in Hilbert scales with $s\ge a$,
the order optimal convergence rates hold for $x_0-x^\dag\in \X_\mu$ with
$s<\mu\le 2s$.
\end{remark}

\begin{remark}
We indicate how Theorem \ref{T1} can be derived from Theorem \ref{T2}. First, we note that
Assumption \ref{A0} (a) implies
$$
\|F(x)-F(z)-F'(z)(x-z)\|\le \frac{1}{2} K_0\|x-z\| \|F'(z) (x-z)\|
$$
for all $x, z\in B_\rho(x^\dag)$. One can then follow the proofs in
Section 3 to show that, if $x_0\in B_\rho(x^\dag)$ and $K_0\|x_0-x^\dag\|$ is sufficiently small,
then these inexact Newton regularization methods are well-defined and
$$
\|x_{n+1}-x^\dag\|\le \|x_n-x^\dag\|, \quad n=0, \cdots, n_\d-1
$$
which implies $x_n\in B_\rho(x^\dag)$ for $0\le n\le n_\d$. By shrinking the ball $B_\rho(x^\dag)$
if necessary, we can derive from Assumption \ref{A0} (a) that there exist two constants
$0<C_0\le C_1<\infty$ such that
\begin{equation}\label{5.2.1}
C_0\|F'(z) h\|\le \|F'(x) h\|\le C_1\|F'(z) h\|, \quad h\in \X
\end{equation}
for all $x, z\in B_\rho(x^\dag)$. This implies that all the operators $F'(x)$ have the
same null space $\N$ as long as $x\in B_\rho(x^\dag)$. By the condition of Theorem \ref{T1} we have
$x_0-x^\dag\in \N^\perp$. By the definition of $\{x_n\}$ we also have
$x_{n+1}-x_n \in \R(F'(x_n)^*)\subset \N^\perp$ for $n=0, \cdots, n_\d-1$. By considering the
operator $G(z):=F(z+x_0)$ if necessary, we may assume $x_0=0$. Therefore $x^\dag, x_n\in \N^\perp$ for
$n=0, \cdots, n_\d$, and we may consider the equation (\ref{1}) on $\N^\perp$. Consequently we may assume
$\N=\{0\}$, i.e. each $F'(x)$ is injective for $x\in B_\rho(x^\dag)$.

Now we introduce the operator $L:=(F'(x^\dag)^* F'(x^\dag))^{-1/2}$ which is clearly
densely defined self-adjoint strictly positive linear operator in $\X$ satisfying
$$
\|x\|^2\le \Theta (L x, x), \quad x\in \D(L).
$$
From (\ref{5.2.1}) it follows that $C_0\|h\|_{-1}\le \|F'(x) h\|\le C_1\|h\|_{-1}$ which implies
Assumption \ref{A1} (a) with $a=1$. Moreover, from Assumption \ref{A0} (b) it follows
for $x,z\in B_\rho(x^\dag) $that
\begin{align*}
\|[F'(x)-F'(z)] \|_{\X_{-1}\to \Y}=\|[F'(x)-F'(z)] L\|_{\X\to \Y}\le K_0\|x-z\|\|F'(z) L\|_{\X\to \Y}.
\end{align*}
Since (\ref{5.2.1}) implies $\|F'(z) L\|_{\X\to \Y}\le C_1$, Assumption \ref{A1} (c) holds with
$b=1$ and $\beta=1$. Since $\R((F'(x^\dag)F'(x^\dag))^\mu)=\X_{2\mu}$,
Theorem \ref{T1} follows immediately from Theorem \ref{T2} with $s=0$.

\end{remark}

\section{\bf Monotonicity of the error}
\setcounter{equation}{0}

We start with a simple consequence of
Assumption \ref{A1} which will be used frequently.

\begin{lemma}\label{L2.0}
Let $F$ satisfy Assumption \ref{A1} and let $x, z\in B_\rho(x^\dag)$. If $t\ge 0$ then
\begin{equation}\label{5.1.1}
\|F(x)-F(z)-F'(z) (x-z)\|\le \frac{1}{1+\beta}K_0 \|x-z\|_t^{\frac{a(1+\beta)-b}{a+t}}
\|x-z\|_{-a}^{\frac{t(1+\beta)+b}{a+t}}.
\end{equation}
If, in addition, $t\ge (a-b)/\beta$, then
\begin{equation}\label{5.1.2}
\|F(x)-F(z)-F'(z) (x-z)\|\le \frac{1}{1+\beta}\gamma^{t\beta+b-a} K_0 \|x-z\|_t^\beta
\|x-z\|_{-a}.
\end{equation}
\end{lemma}

\begin{proof} From Assumption \ref{A1} (c) and the identity
$$
F(x)-F(z)-F'(z)(x-z)=\int_0^1 \left[F'(z+t(x-z))-F'(z)\right] (x-z) dt
$$
it follows immediately that
\begin{equation}\label{5.1.3}
\|F(x)-F(z)-F'(z)(x-z)\|\le \frac{1}{1+\beta} K_0\|x-z\|^\beta \|x-z\|_{-b}.
\end{equation}
With the help of the interpolation inequality (\ref{inter}) we have
$$
\|x-z\|\le \|x-z\|_t^{\frac{a}{a+t}} \|x-z\|_{-a}^{\frac{t}{a+t}} \quad
\mbox{and}\quad \|x-z\|_{-b} \le \|x-z\|_t^{\frac{a-b}{a+t}} \|x-z\|_{-a}^{\frac{t+b}{a+t}}.
$$
This together with (\ref{5.1.3}) gives (\ref{5.1.1}). If, in addition, $t\ge (a-b)/\beta$,
then we have $[t(1+\beta)+b]/(a+t)\ge 1$. Thus, by using $\|x-z\|_{-a}\le \gamma^{a+t} \|x-z\|_t$
which follows from the embedding (\ref{embed}), we can derive (\ref{5.1.2}) immediately from (\ref{5.1.1}).
\hfill $\Box$
\end{proof}

In this section we will use the ideas from \cite{H97,HNS96,LR2010} to show that the four inexact Newton
regularization methods in Hilbert scales stated in Theorem \ref{T2} are well-defined and for the error term
$$
e_n:=x_n-x^\dag
$$
there holds $\|e_{n+1}\|_s\le \|e_n\|_s$ for $n=0, \cdots, n_\d-1$. We will use the notation
$$
T:=F'(x^\dag), \quad T_n:=F'(x_n), \quad A:=T L^{-s} \quad \mbox{and} \quad A_n:=T_n L^{-s}.
$$
It follows easily from the definition (\ref{5.7.1}) of $\{u_n(t)\}$ that
\begin{equation}\label{g1.1}
u_n(t)=L^{-s} g_t(A_n^* A_n) A_n^* \left(y^\d-F(x_n)\right)
\end{equation}
and
\begin{equation}\label{r1}
y^\d-F(x_n)-T_n u_n(t)=r_t(A_nA_n^*) \left(y^\d-F(x_n)\right),
\end{equation}
where $r_t(\la):=1-\la g_t(\la)$ denotes the residual function associated with $g_t$.
For the spectral filter functions given in (\ref{5.7.2}), it is easy to see that
$\lim_{t\rightarrow \infty} r_t(\la)=0$ for each $\la>0$. This implies that
\begin{align}\label{5.7.4}
\lim_{t\rightarrow \infty} \|y^\d-F(x_n)-T_n u_n(t)\|&=\|P_{\R(A_n)^\perp}(y^\d-F(x_n))\|,
\end{align}
where $P_{\R(A_n)^\perp}$ denotes the orthogonal projection of $\Y$ onto $\R(A_n)^\perp$,
the orthogonal complement of the range $\R(A_n)$ of $A_n$.

\begin{lemma}\label{L2.1}
Let $F$ satsify Assumption \ref{A1} with $s\ge (a-b)/\beta$, let $\tau>1$ and $0<\eta<1$
satisfy $\tau \eta >1$, and let $x_0\in \D(F)$ be such that $\gamma^s \|e_0\|_s\le \rho$.
Assume that $K_0\|e_0\|_s^\beta$ is sufficiently small. If $\|y^\d-F(x_n)\|>\tau \d$
and $\|e_n\|_s\le \|e_0\|_s$, then $t_n$ is well-defined and $t_n\ge c_0$ for some constant
$c_0>0$ independent of $n$ and $\d$.
\end{lemma}

\begin{proof}
From (\ref{embed}) and the given conditions it follows that
$\|e_n\|\le \gamma^s \|e_n\|_s \le \gamma^s \|e_0\|_s\le \rho$ which implies $x_n\in B_\rho(x^\dag)$.
Since $\|e_n\|_s\le \|e_0\|_s<\infty$ implies $L^s e_n\in \X$, we have
\begin{align*}
\|P_{\R(A_n)^\perp} (y^\d-F(x_n))\| &\le \|y^\d-F(x_n)+A_n L^s e_n\|=\|y^\d-F(x_n)+ T_n e_n\|.
\end{align*}
In order to show that $t_n$ is well-defined, in view of (\ref{5.7.4}) it suffices to show
\begin{align}\label{2.10}
\|y^\d-F(x_n) +T_n e_n\| < \eta \|y^\d-F(x_n)\|.
\end{align}
Since $s\ge (a-b)/\beta$, we can use (\ref{1.2}) and (\ref{5.1.2}) in Lemma \ref{L2.0} to derive
$$
\|y^\d-F(x_n) +T_n e_n\| \le \d+ \frac{1}{1+\beta} \gamma^{s\beta+b-a} K_0\|e_n\|_s^\beta \|e_n\|_{-a}
$$
Now by using Assumption \ref{A1} (a), $\|e_n\|_s\le \|e_0\|_s$ and  $\tau \d< \|y^\d-F(x_n)\|$, we obtain
with $C=\gamma^{s\beta+b-a}/[(1+\beta) m]$ that
\begin{align*}
\|y^\d-F(x_n) +T_n e_n\| &\le  \frac{1}{\tau} \|y^\d-F(x_n)\|
+ C K_0 \|e_0\|_s^\beta \|T_n e_n\|\\
&\le \left(\frac{1}{\tau}+ C K_0\|e_0\|_s^\beta\right) \|y^\d-F(x_n)\|\\
&\quad \, + C K_0\|e_0\|_s^\beta \|y^\d-F(x_n) +T_n e_n\|.
\end{align*}
Since $\tau \eta>1$, we therefore obtain (\ref{2.10}) if $K_0\|e_0\|_s$ is sufficiently small.

For the inner scheme defined by Landweber iteration or the implicit iteration in Hilbert scales, it is obvious that
$t_n$ is an integer with $t_n\ge 1$.  For the inner scheme defined by the asymptotic regularization
or Tikhonov regularization in Hilbert scales, we have
$$
\eta \|y^\d-F(x_n)\|=\|y^\d-F(x_n)-T_n u_n(t_n)\|=\|r_{t_n}(A_nA_n^*) (y^\d-F(x_n))\|
$$
where $r_t(\la)=e^{-t \la}$ or $r_t(\la)=(1+t\la)^{-1}$. Since $\|A_n\|\le 1$, we can obtain
either $e^{-t_n}\le \eta$ or $(1+t_n)^{-1} \le \eta$. Therefore $t_n\ge \log(1/\eta)$ or $t_n\ge 1/\eta-1$.
\hfill $\Box$
\end{proof}

\begin{lemma}\label{L2.2}
Let $F$ satisfy Assumption \ref{A1} with $s\ge (a-b)/\beta$, let $\tau>2$ and $0<\eta<1$
be such that $ \tau\eta>2$, and let $x_0\in \D(F)$ be such that $\gamma^s\|e_0\|_s\le \rho$.
If $K_0\|e_0\|_s^\beta$ is sufficiently small, then the four inexact Newton regularization methods
in Hilbert scales stated in Theorem \ref{T2} are well-defined and terminate
after $n_\d<\infty$ iterations, and
\begin{equation}\label{4.30.1}
\sum_{n=0}^{n_\d-1} t_n \|y^\d-F(x_n)\|^2 \le C_2 \|e_0\|_s^2
\end{equation}
for some constant $C_2>0$. Moreover
\begin{equation}\label{2.11}
\|x_{n+1}-x^\dag\|_s\le \|x_n-x^\dag\|_s
\end{equation}
for $n=0, \cdots, n_\d-1$.
\end{lemma}

\begin{proof}
We will prove this result for the four inexact Newton methods case by case.

(a) We first consider the inexact Newton method with inner scheme defined by Landweber iteration
in Hilber scales. We first show the monotonicity (\ref{2.11}).
We may assume $n_\d\ge 1$. Let $0\le n<n_\d$ and assume that $\|e_n\|_s\le \|e_0\|_s$.
By the definition of $n_\d$ we have $\|y^\d-F(x_n)\|>\tau \d$.
It follows from Lemma \ref{L2.1} that $t_n$ is a well-defined positive integer.
Let $u_{n,k}:=u_n(k)$ for each integer $k$. Then $u_{n,0}=0$ and
$$
u_{n, k}=u_{n,k-1}+ L^{-2s} T_n^* \left(y^\d-F(x_n)- T_n u_{n,k-1}\right)
$$
for $k=1, \cdots, t_n$. Recall that $x_{n+1}=x_n+u_{n, t_n}$. Therefore,
in order to show $\|e_{n+1}\|_s\le \|e_n\|_s$, it suffices to show
\begin{equation}\label{2.12}
\|e_n+u_{n, k}\|_s\le \|e_n+u_{n, k-1}\|_s, \qquad k=1, \cdots, t_n.
\end{equation}
We set $z_{n,k}=y^\d-F(x_n)-T_n u_{n, k}$. Then $u_{n,k}-u_{n, k-1}=L^{-2s} T_n^* z_{n, k-1}$ and thus
\begin{align*}
\|e_n+u_{n, k}&\|_s^2 -\|e_n+u_{n, k-1}\|_s^2\\
&=2(e_n+u_{n, k-1}, u_{n, k}-u_{n, k-1})_s+\|u_{n, k}-u_{n, k-1}\|_s^2\\
&=(u_{n, k}-u_{n, k-1}, u_{n, k}+u_{n, k-1}+2 e_n)_s\\
&=\left(z_{n, k-1}, T_n ( u_{n, k} + u_{n, k-1} + 2 e_n)\right).
\end{align*}
According to the definition of $z_{n,k}$ one can see
$$
T_n ( u_{n, k} + u_{n, k-1} + 2 e_n)=-z_{n,k} -z_{n, k-1} +2(y^\d-F(x_n)+ T_n e_n).
$$
Therefore
\begin{align*}
\|e_n+ &u_{n, k}\|_s^2 -\|e_n+u_{n, k-1}\|_s^2\\
&=-(z_{n, k-1}, z_{n,k})-\|z_{n, k-1}\|^2 + 2 (z_{n,k-1}, y^\d-F(x_n) +T_n e_n).
\end{align*}
Observing that (\ref{r1}) and $r_t(\la)=(1-\la)^{[t]}$ imply $z_{n,k}=(I-A_nA_n^*)^k (y^\d-F(x_n))$,
we have $(z_{n,k-1}, z_{n,k})\ge 0$. Hence
\begin{align*}
\|e_n+& u_{n, k}\|_s^2 -\|e_n+u_{n, k-1}\|_s^2\\
&\le -\|z_{n,k-1}\|\left(\|z_{n,k-1}\|- 2 \|y^\d-F(x_n) +T_n e_n\|\right).
\end{align*}
Since $\tau \eta>2$, we can pick $0<\eta_0<\eta/2$ with $\tau \eta_0>1$. By using
Assumption \ref{A1}, $\tau \d< \|y^\d-F(x_n)\|$ and $\|e_n\|_s\le \|e_0\|_s$,
we can derive as in the proof of Lemma \ref{L2.1} that if $K_0\|e_0\|_s^\beta$ is
sufficiently small then
\begin{align*}
\|y^\d-F(x_n) + T_n e_n\|\le  \eta_0 \|y^\d-F(x_n)\|.
\end{align*}
On the other hand, by the definition of $t_n$
we have $\|z_{n,k-1}\|>\eta \|y^\d-F(x_n)\|$. Therefore
\begin{align}\label{2.13}
\|e_n+u_{n, k}\|_s^2& -\|e_n+u_{n, k-1}\|_s^2\le -\varepsilon_0 \|y^\d-F(x_n)\|^2,
\end{align}
where $\varepsilon_0:=\eta(\eta-2\eta_0)>0$. This in particular implies (\ref{2.12})
and hence $\|e_{n+1}\|_s\le \|e_n\|_s$. An induction argument
then shows the monotonicity result (\ref{2.11}).

Moreover, it follows from (\ref{2.13}) that
\begin{align*}
\|e_{n+1}\|_s^2 -\|e_n\|_s^2 &=\sum_{k=1}^{t_n} \left(\|e_n+u_{n, k}\|_s^2-\|e_n+u_{n, k-1}\|_s^2\right)\\
&\le - \varepsilon_0 t_n \|y^\d-F(x_n)\|^2.
\end{align*}
Consequently
$$
\varepsilon_0 \sum_{n=0}^{n_\d-1} t_n \|y^\d-F(x_n)\|^2
\le \|e_0\|_s^2-\|e_{n_\d}\|_s^2 \le \|e_0\|_s^2 <\infty
$$
which shows (\ref{4.30.1}). Since $t_n\ge 1$ and $\|y^\d-F(x_n)\|>\tau\d$ for $0\le n<n_\d$,
one can see that $n_\d$ must be finite.

(b) For the inexact Newton method with inner scheme defined by the implicit iteration
in Hilbert scales, all $t_n$ must be positive integer and with the notation $u_{n,k}:=u_n(k)$
we have $u_{n,0}=0$ and
$$
u_{n, k}=u_{n,k-1}+(L^{2s} + T_n^* T_n)^{-1} T_n^* \left(y^\d-F(x_n) -T_n u_{n,k-1}\right).
$$
Let $z_{n,k}:=y^\d-F(x_n)-T_n u_{n, k}$. We have from (\ref{r1}) and $r_t(\la)=(1+\la)^{-[t]}$
that $z_{n,k}=(I+A_nA_n^*)^{-1} z_{n,k-1}$ and $u_{n, k}-u_{n,k-1}=L^{-2s} T_n^* z_{n,k}$. Thus
\begin{align*}
\|e_n+u_{n,k}&\|_s^2-\|e_n+u_{n,k-1}\|_s^2\\
&=(u_{n,k}-u_{n,k-1}, u_{n,k}+u_{n,k-1}+2 e_n)_s\\
&=(z_{n,k}, T_n(u_{n,k}+u_{n,k-1}+2 e_n))\\
&=(z_{n,k}, -z_{n,k}-z_{n,k-1} +2(y^\d-F(x_n) +T_n e_n)).
\end{align*}
Note that $(z_{n,k}, z_{n,k-1})\ge \|z_{n,k}\|^2$. We then obtain
\begin{align*}
\|e_n+u_{n,k}\|_s^2-\|e_n+u_{n,k-1}\|_s^2
&\le -2\|z_{n,k}\|\left(\|z_{n,k}\|-\|y^\d-F(x_n) +T_n e_n\|\right).
\end{align*}
By using $\|A_n\|\le 1$ and the definition of $t_n$, we have
$$
\|z_{n,k}\|\ge \frac{1}{2}\|z_{n,k-1}\| \ge \frac{1}{2}\eta \|y^\d-F(x_n)\|, \quad k=1,\cdots, t_n.
$$
Since $\tau \eta>2$, we can obtain
\begin{align*}
\|e_n+u_{n,k}\|_s^2-\|e_n+u_{n,k-1}\|_s^2
&\le -\frac{1}{2} \eta (\eta-2\eta_0) \|y^\d-F(x_n)\|^2.
\end{align*}
for $k=1, \cdots, t_n$ when $K_0\|e_0\|_s$ is sufficiently small, where $0<\eta_0<\eta/2$ is such that $\tau \eta_0>1$.
This together with an induction argument implies (\ref{4.30.1}) and (\ref{2.11}).

(c) For the inexact Newton method with inner scheme defined by the asymptotic regularization in
Hilbert scales, $u_n(t)$ is the solution of the initial value problem
\begin{align*}
\frac{d}{d t} u_n(t) &=L^{-2s} T_n^* \left(y^\d-F(x_n)-T_n u_n(t)\right), \quad t>0,\\
u_n(0)&=0.
\end{align*}
Therefore, with $z_n(t):= y^\d-F(x_n)-T_n u_n(t)$ we have
\begin{align*}
\frac{d}{d t} \|e_n+u_n(t)\|_s^2&=2 \left( \frac{d}{d t} u_n(t), e_n +u_n(t)\right)_s
=2\left(z_n(t), T_n(e_n+u_n(t))\right)\\
&=2(z_n(t), -z_n(t)+ y^\d-F(x_n)+T_n e_n)\\
&\le -2\|z_n(t)\| \left(\|z_n(t)\|-\|y^\d-F(x_n)+T_n e_n\|\right).
\end{align*}
According to the definition of $t_n$ we have $\|z_n(t_n)\|=\eta\|y^\d-F(x_n)\|$ and
$\|z_n(t)\|> \eta\|y^\d-F(x_n)\|$ for $0\le t\le t_n$. Since $\tau \eta>1$, we therefore obtain
$$
\frac{d}{d t} \|e_n +u_n(t)\|_s^2\le -2\eta(\eta-\eta_0) \|y^\d-F(x_n)\|^2, \quad 0<t\le t_n
$$
if $K_0\|e_0\|_s^\beta$ is sufficiently small, where $0<\eta_0<\eta$ is such that $\tau \eta_0>1$.
In view of $u_n(0)=0$ and $x_{n+1}=x_n+u_n(t_n)$, we obtain
$$
\|e_{n+1}\|_s^2-\|e_n\|_s^2 \le -2\eta(\eta-\eta_0) t_n\|y^\d-F(x_n)\|^2.
$$
This implies (\ref{4.30.1}) and (\ref{2.11}) immediately.

(d)
For the inexact Newton method with inner scheme defined by Tikhonov regularization,
we have
$$
u_n(t)=\left(t^{-1}  L^{2s} + T_n^* T_n\right)^{-1} T_n^* (y^\d-F(x_n)).
$$
We first observe that
\begin{align*}
\|e_{n+1}\|_s^2 -\|e_n\|_s^2
&\le 2 \|x_{n+1}-x_n\|_s^2 +2(x_{n+1}-x_n, e_n)_s\\
&=2(x_{n+1}-x_n, x_{n+1}-x_n +e_n)_s.
\end{align*}
Let $z_n=y^\d-F(x_n)-T_n (x_{n+1}-x_n)$. We have from (\ref{r1}) and $r_t(\la)=(1+t \la)^{-1}$ that
$u_n(t_n)= t_n L^{-2s} T_n^* z_n$ and hence $x_{n+1}-x_n=t_n L^{-2s} T_n^* z_n$. Therefore
\begin{align*}
\|e_{n+1}\|_s^2 -\|e_n\|_s^2 &\le 2 t_n \left(z_n, T_n (x_{n+1}-x_n +e_n) \right)\\
&=2 t_n \left(z_n, -z_n +(y^\d-F(x_n) +T_n e_n)\right)\\
&\le -2t_n \|z_n\|\left(\|z_n\|-\|y^\d-F(x_n)+T_n e_n\|\right).
\end{align*}
By the definition of $t_n$ we have $\|z_n\|= \eta \|y^\d-F(x_n)\|$.
Since $\tau \eta>1$, we can obtain
\begin{align*}
\|e_{n+1}\|_s^2 -\|e_n\|_s^2
&\le -2\eta(\eta-\eta_0) t_n \|y^\d-F(x_n)\|^2
\end{align*}
if $K_0\|e_0\|_s^\beta$ is sufficiently small, where $0<\eta_0<\eta$ is such that $\tau \eta_0>1$.
This implies (\ref{4.30.1}) and (\ref{2.11}). \hfill $\Box$
\end{proof}

\begin{remark} \label{R3.1}
The inequality (\ref{4.30.1}) will find its use in the proof of Lemma \ref{P1}. From (\ref{4.30.1}), $t_n\ge c_0>0$, and the fact
$\|y^\d-F(x_n)\|\ge \tau \d$ for $0\le n<n_\d$, it follows easily that
$n_\d=O(\d^{-2})$ which gives only a rough estimate on the number of outer iterations. However, we should point out
that the inexact Newton iterations in Hilbert scales in fact terminate after
$n_\d=O(1+|\log \d|)$ outer iterations. This can be confirmed by using the fact
\begin{align}\label{4.29.0}
\eta &\|y^\d-F(x_n)\| \ge \|y^\d-F(x_n)-T_n(x_{n+1}-x_n)\|, \quad 0\le n<n_\d
\end{align}
which follows from the definition of $t_n$ and $x_{n+1}=x_n+u_n(t_n)$. To see this, by using
(\ref{5.1.2}) in Lemma \ref{L2.0} we have
\begin{align*}
\|F(x_{n+1})-F(x_n)-T_n (x_{n+1}-x_n)\| \le \frac{\gamma^{s\beta+b-a}}{1+\beta}  K_0\|x_{n+1}-x_n\|_s^\beta\|x_{n+1}-x_n\|_{-a}.
\end{align*}
Since (\ref{2.11}) implies $\|x_{n+1}-x_n\|_s\le \|e_{n+1}\|_s+\|e_n\|_s\le 2\|e_0\|_s$,
from Assumption \ref{A1} (a) we have with $C:=2^\beta\gamma^{s\beta+b-a}/[(1+\beta)m]$ that
\begin{align*}
\|F(x_{n+1})-F(x_n)-T_n (x_{n+1}-x_n)\| & \le C K_0\|e_0\|_s^\beta \|T_n (x_{n+1}-x_n)\|.
\end{align*}
Therefore, if $K_0\|e_0\|_s^\beta$ is sufficiently small, then there holds
$\|T_n (x_{n+1}-x_n)\|\le 2 \|F(x_{n+1})-F(x_n)\|$ and consequently
\begin{align}\label{4.29}
\|F(x_{n+1})-F(x_n)&-T_n(x_{n+1}-x_n)\| \le 2C K_0\|e_0\|_s^\beta \|F(x_{n+1})-F(x_n)\|.
\end{align}
Combining this with (\ref{4.29.0}) yields
$$
\eta \|y^\d-F(x_n)\|\ge \|y^\d-F(x_{n+1})\| - 2 C K_0\|e_0\|_s^\beta \|F(x_{n+1})-F(x_n)\|.
$$
Considering $\eta<1$, this in particular implies that if $K_0\|e_0\|_s^\beta$ is sufficiently small then
$$
\frac{\|y^\d-F(x_{n+1})\|}{\|y^\d-F(x_n)\|}
\le \frac {\eta+ 2 C K_0\|e_0\|_s^\beta}{1-2 C K_0\|e_0\|_s^\beta}
\le \frac{1+\eta}{2}<1.
$$
Therefore for all $n=0, \cdots, n_\d$ there holds
$$
\|y^\d-F(x_n)\|\le \left(\frac{1+\eta}{2}\right)^n \|y^\d-F(x_0)\|.
$$
By taking $n=n_\d-1$ and using $\|y^\d-F(x_{n_\d-1})\|\ge \tau \d$ we obtain
$\tau \d\le \left(\frac{1+\eta}{2}\right)^{n_\d-1} \|y^\d-F(x_0)\|$ which shows that $n_\d=O(1+|\log\d|)$.
\end{remark}

\section{\bf Proof of Theorem \ref{T2}}
\setcounter{equation}{0}

In this section we will show the order optimality of the four inexact Newton method in Hilbert scales
stated in Theorem \ref{T2}. For simplicity of further exposition, we will always use $C$ to denote a generic
constant independent of $\d$ and $n$, we will also use the
convention $\Phi\lesssim \Psi$ to mean that $\Phi\le C \Psi$ for
some generic constant $C$ when the explicit expression of $C$ is not important.
Furthermore, we will use $\Phi\sim \Psi$ to mean that $\Phi\lesssim \Psi$ and $\Psi\lesssim \Phi$.

\begin{lemma}\label{L2.3}
Under the same conditions in Lemma \ref{L2.2}, there holds
$$
\|y^\d-F(x_n)\|\lesssim \|y^\d-F(x_{n+1})\|, \qquad n=0, \cdots, n_\d-1.
$$
\end{lemma}
\begin{proof}
We first claim that there is a constant $c_1>0$ such that
\begin{align}\label{5.7.6}
c_1\|y^\d-F(x_n)\| \le \|y^\d-F(x_n)- T_n (x_{n+1}-x_n)\|.
\end{align}
This is clear from the definition of $t_n$ when the inner scheme is defined by
Tikhonov regularization or the asymptotic regularization. When the inner scheme
is defined by Landweber iteration, we have $r_t(\la)=(1-\la)^{[t]}$.
According to the definition of $t_n$ and (\ref{r1}), we have
\begin{align*}
\eta \|y^\d-F(x_n)\| &\le \|y^\d-F(x_n)-T_n u_n(t_n-1)\|\\
&=\|(I-A_n A_n^*)^{t_n-1} (y^\d-F(x_n))\|.
\end{align*}
Since $\|A_n\|\le \Theta<1$, we have $\|(I-A_n A_n^*)^{-1}\|\le (1-\Theta^2)^{-1}$. Therefore, using
(\ref{r1}) again it follows
\begin{align*}
(1-\Theta^2) \eta\|y^\d-F(x_n)\|&\le  \|(I- A_nA_n^*)^{t_n} (y^\d-F(x_n))\| \nonumber\\
&= \|y^\d-F(x_n)- T_n (x_{n+1}-x_n)\|
\end{align*}
which shows (\ref{5.7.6}) with $c_1=(1-\Theta^2)\eta$.  When the inner scheme is defined by the
implicit iteration, we have $r_t(\la)=(1+\la)^{-[t]}$.  Thus it follows from (\ref{r1}) and $\|A_n\|\le 1$ that
\begin{align*}
\eta \|y^\d-F(x_n)\| &\le \|(I+A_nA_n^*)^{-t_n+1} (y^\d-F(x_n)) \| \\
&\le 2 \|(I+A_nA_n^*)^{-t_n} (y^\d-F(x_n)) \| \\
&= 2 \|y^\d-F(x_n) -T_n(x_{n+1}-x_n)\|
\end{align*}
which shows (\ref{5.7.6}) with $c_1=\eta/2$.

The combination of (\ref{5.7.6}) and (\ref{4.29}) gives
\begin{align*}
c_1&\|y^\d-F(x_n)\|\\
&\le  \|y^\d-F(x_{n+1})\| + C K_0\|e_0\|_s^\beta \|F(x_{n+1})-F(x_n)\|\\
&\le  \|y^\d-F(x_{n+1})\| + C K_0\|e_0\|_s^\beta \left(\|y^\d-F(x_{n+1})\|+\|y^\d-F(x_n)\|\right).
\end{align*}
This shows the result if $K_0\|e_0\|_s^\beta$ is sufficiently small.
\hfill $\Box$
\end{proof}

For the spectral filter functions defined by (\ref{5.7.2}), we have shown in \cite{JT2011} that for any
sequence of positive numbers $\{t_n\}$ there hold
\begin{align}
0\le \la^\nu \prod_{k=j}^{n-1} r_{t_k}(\la)&\le (s_n-s_j)^{-\nu},  \label{g1}\\
0\le \la^\nu g_{t_j}(\la) \prod_{k=j+1}^{n-1} r_{t_k}(\la)&\le
t_j (s_n-s_j)^{-\nu} \label{g2}
\end{align}
and
\begin{equation}\label{g3}
0\le \la^\nu \sum_{i=0}^{n-1} g_{t_i}(\la) \prod_{k=i+1}^{n-1} r_{t_k}(\la)\le s_n^{1-\nu}
\end{equation}
for $0\le \nu\le 1$,  $0\le \la \le 1$ and $j=0, 1, \cdots, n-1$, where $\{s_n\}$ is defined by
\begin{equation}\label{s1}
s_0=0 \qquad \mbox{and} \qquad s_n=\sum_{j=0}^{n-1} t_j  \quad \mbox{for } n=1,2, \cdots.
\end{equation}
Moreover, we have the following crucial estimate.

\begin{lemma}\label{L10}
Let $F$ satisfy Assumption \ref{A1}, let $\{g_t\}$ be defined by (\ref{5.7.2}) and $r_t(\la)=1-\la g_t(\la)$,
and let $\{t_n\}$ be a sequence of positive numbers with $\{s_n\}$
defined by (\ref{s1}). Let $A=F'(x^\dag) L^{-s}$ and for any $x\in B_\rho(x^\dag)$ let $A_x=F'(x) L^{-s}$.
Then for $-\frac{b+s}{2(a+s)} \le \nu\le 1/2$ there holds
\begin{align*}
&\left\| (A^*A)^\nu \prod_{k=j+1}^{n-1} r_{t_k}(A^*A) \left[ g_{t_j}(A^*A)A^*
-g_{t_j}(A_x^*A_x)A_x^*\right] \right\| \nonumber\\
&\qquad\qquad \qquad\qquad\qquad\qquad
\lesssim t_j (s_n-s_j)^{-\nu-\frac{b+s}{2(a+s)}} K_0\|x-x^\dag\|^\beta
\end{align*}
for $j=0, 1, \cdots, n-1$.
\end{lemma}

\begin{proof}
We refer to \cite[Lemma 2]{JT2011} in which similar estimates have been derived for a general class of
spectral filter functions. \hfill $\Box$
\end{proof}

We also need the following estimate concerning the sums of suitable types
which will occur in the convergence analysis.

\begin{lemma}\label{L2}
Let $\{t_n\}$ be a sequence of numbers satisfying $t_n\ge c_2>0$, and let
$s_n$ be defined by (\ref{s1}). Let $p\ge 0$ and $q\ge0$ be two numbers. Then
$$
\sum_{j=0}^{n-1}  t_j (s_n-s_j)^{-p} s_{j+1}^{-q}\le C_3
s_n^{1-p-q} \left\{\begin{array}{lll}
1, & \max\{p,q\}<1,\\
\log (1+s_n), & \max\{p, q\}=1,\\
s_n^{\max\{p, q\}-1}, & \max\{p, q\}>1,
\end{array}\right.
$$
where $C_3$ is a constant depending only on $p$, $q$ and $c_2$.
\end{lemma}

\begin{proof}
This is essentially contained in \cite[Lemma 4.3]{H2010} and its proof. A
simplified proof can be found in \cite[Lemma 3]{JT2011}. \hfill $\Box$
\end{proof}

Now we are ready to give the crucial estimates on $\|e_n\|_\mu$ and
$\|T e_n\|$ for $0\le n<n_\d$. We will exploit the ideas developed in
\cite{H2010,Jin2011,JT2011}.

\begin{lemma}\label{P1}
Let $F$ satisfy Assumption \ref{A1} with $s\ge (a-b)/\beta$, let $\tau>2$ and $0< \eta<1$ be such that
$\tau \eta>2$, let $x_0\in \D(F)$ satisfy $\gamma^s \|e_0\|_s\le \rho$.
If $e_0 \in \X_\mu$ for some $s<\mu \le b+2s$ and if
$K_0\|e_0\|_\mu^\beta$ is sufficiently small, then there exists a constant $C_*>0$ such that
\begin{align*}
\|e_n\|_\mu  \le C_* \|e_0\|_\mu  \quad \mbox{and} \quad
\|T e_n\| \le C_* \|e_0\|_\mu (1+s_n)^{-\frac{a+\mu}{2(a+s)}}
\end{align*}
for all $n=0, \cdots, n_\d-1$.
\end{lemma}

\begin{proof} Since $s<\mu \le b+2s$, from (\ref{2.3}) we have
$\|e_n\|_\mu\sim \|(A^*A)^{\frac{s-\mu}{2(a+s)}} L^s e_n\|$. Therefore, it suffices to show
that there exists a constant $C_*>0$ such that
\begin{equation}\label{402}
\|(A^*A)^{\frac{s-\mu}{2(a+s)}} L^s e_n\|\le C_* \|e_0\|_\mu
\quad \mbox{and}\quad \|T e_n\| \le C_* \|e_0\|_\mu (1+s_n)^{-\frac{a+\mu}{2(a+s)}}
\end{equation}
for all $n=0, \cdots, n_\d-1$. We will show (\ref{402}) by induction. By using (\ref{2.3}) and Assumption \ref{A1} (b)
we have
$$
\|(A^*A)^{\frac{s-\mu}{2(a+s)}} L^s e_0\|\le \overline{c}(\frac{s-\mu}{a+s}) \|e_0\|_\mu
$$
and
$$
\|T e_0\|=\|(A^*A)^{1/2} L^s e_0\|\le \|(A^*A)^{\frac{s-\mu}{2(a+s)}} L^s e_0\|\le \overline{c}(\frac{s-\mu}{a+s}) \|e_0\|_\mu.
$$
Therefore (\ref{402}) with $n=0$ holds for $C_*\ge \overline{c}(\frac{s-\mu}{a+s})$.
Now we assume that (\ref{402}) is true for all $0\le n< l$ for some $0<l<n_\d$ and want to show that
it is  also true for $n=l$.

From the equation (\ref{g1.1}) and $x_{n+1}=x_n+u_n(t_n)$ it follows that
\begin{align*}
e_{n+1}&=e_n + L^{-s} g_{t_n}(A_n^*A_n) A_n^* \left(y^\d-F(x_n)\right) \\
&=L^{-s} r_{t_n}(A^*A) L^s e_n +L^{-s} g_{t_n}(A^*A) A^* (y^\d-F(x_n)+ T e_n) \nonumber\\
&\quad\, +L^{-s} \left[g_{t_n}(A_n^*A_n)A_n^* -g_{t_n}(A^*A)A^* \right] (y^\d-F(x_n)).
\end{align*}
By induction on this equation we obtain
\begin{align}\label{20}
e_l &=L^{-s} \prod_{j=0}^{l-1} r_{t_j}(A^* A) L^{s} e_0 +L^{-s} \sum_{j=0}^{l-1} \prod_{k=j+1}^{l-1} r_{t_k}(A^*A) g_{t_j}(A^*A) A^*(y^\d-y)\nonumber\\
&\quad\,  +L^{-s} \sum_{j=0}^{l-1} \prod_{k=j+1}^{l-1} r_{t_k}(A^*A) g_{t_j}(A^*A) A^*\left(y-F(x_j)+ T e_j\right) \nonumber\\
&\quad\, +L^{-s} \sum_{j=0}^{l-1} \prod_{k=j+1}^{l-1} r_{t_k}(A^*A)
\left[g_{t_j}(A_j^*A_j)A_j^*-g_{t_j}(A^*A)A^*\right] \left(y^\d-F(x_j)\right).
\end{align}
By multiplying  (\ref{20}) by $T:=F'(x^\dag)$, noting that $A=T L^{-s}$, and using the identity
\begin{equation*}
1-\la \sum_{j=0}^{l-1} g_{t_j}(\la) \prod_{k=j+1}^{l-1} r_{t_k}(\la)=\prod_{j=0}^{l-1} r_{t_j}(\la)
\end{equation*}
which follows from the relation $r_t(\la)=1-\la g_t(\la)$, we can obtain
\begin{align}\label{21}
T e_l& =A \prod_{j=0}^{l-1} r_{t_j}(A^*A) L^s e_0
 +\left[I-\prod_{j=0}^{l-1} r_{t_j}(AA^*)\right](y^\d-y) \nonumber\\
&\quad\, + \sum_{j=0}^{l-1} \prod_{k=j+1}^{l-1} r_{t_k}(AA^*) g_{t_j}(AA^*) AA^*\left(y-F(x_j)+ T e_j\right) \nonumber\\
&\quad\, + \sum_{j=0}^{l-1} A \prod_{k=j+1}^{l-1}
r_{t_k}(A^*A) \left[g_{t_j}(A_j^*A_j)A_j^*-g_{t_j}(A^*A)A^*\right] (y^\d-F(x_j)).
\end{align}
Since $e_0\in \X_\mu$ with $s<\mu\le b+2s$, by using (\ref{2.3}), (\ref{g1}), (\ref{g2}), (\ref{g3})
and Lemma \ref{L10} we can derive from (\ref{20}) that
\begin{align}\label{e29}
\|(&A^*A)^{\frac{s-\mu}{2(a+s)}} L^s e_l\| \nonumber\\
&\quad \le  c_3 \|e_0\|_\mu + s_l^{\frac{a+\mu}{2(a+s)}} \d
+ \sum_{j=0}^{l-1} t_j (s_l-s_j)^{-\frac{a+2s-\mu}{2(a+s)}} \|y-F(x_j)+T e_j\| \nonumber\\
&\quad  + C\sum_{j=0}^{l-1} t_j (s_l-s_j)^{-\frac{b+2s-\mu}{2(a+s)}} K_0\|e_j\|^\beta \|y^\d-F(x_j)\|,
\end{align}
where $c_3=\overline{c}(\frac{\mu-s}{a+s})$ and $C$ is a generic constant independent of $l$ and $\d$.

Next by using again $e_0\in \X_\mu$ with $s<\mu\le b+2s$, (\ref{2.3}) and (\ref{g1}), we can obtain
\begin{align*}
\left\|A \prod_{j=0}^{l-1} r_{t_j}(A^*A) L^s e_0\right\|
&\le \left\|A \prod_{j=0}^{l-1} r_{t_j}(A^*A) (A^*A)^{\frac{\mu-s}{2(a+s)}}\right\|
\left\|(A^*A)^{-\frac{\mu-s}{2(a+s)}} L^s e_0\right\|\\
&\le c_3 \sup_{0\le \la\le 1} \left(\la^{\frac{a+\mu}{2(a+s)}} \prod_{j=0}^{l-1} r_{t_j}(\la)\right) \|e_0\|_\mu\\
&\le c_3 s_l^{-\frac{a+\mu}{2(a+s)}} \|e_0\|_\mu.
\end{align*}
Therefore, it follows from (\ref{21}), (\ref{g2}) and Lemma \ref{L10} that
\begin{align}\label{333}
\|T e_l\| &\le c_3 s_l^{-\frac{a+\mu}{2(a+s)}} \|e_0\|_\mu +\d
+ \sum_{j=0}^{l-1} t_j (s_l-s_j)^{-1} \|y-F(x_j)+T e_j\| \nonumber\\
& \quad \, + C \sum_{j=0}^{l-1} t_j (s_l-s_j)^{-\frac{b+a+2s}{2(a+s)}} K_0\|e_j\|^\beta \|y^\d-F(x_j)\|.
\end{align}

We first use (\ref{333}) to derive the desired estimate for $\|Te_l\|$.
According to the relation $\|e_j\|_\mu \sim \|(A^*A)^{\frac{s-\mu}{2(a+s)}} L^s e_j\|$,
we have from the induction hypotheses that
\begin{equation}\label{induction}
\|e_j\|_\mu \lesssim \|e_0\|_\mu \quad \mbox{and}
\quad \|T e_j\|\lesssim \|e_0\|_\mu (1+s_j)^{-\frac{a+\mu}{2(a+s)}}, \quad 0\le j\le l-1.
\end{equation}
We need to estimate the terms
$$
\|e_j\|, \quad \|y^\d-F(x_j)\|\quad \mbox{and}\quad \|y-F(x_j)+T e_j\|, \qquad  0\le j\le l-1.
$$
For each term we will give two types of estimates, one is true for all $0\le j\le l-1$ and the other is true for
$0\le j<l-1$.

By using (\ref{5.1.2}) in Lemma \ref{L2.0}, Assumption \ref{A1} (a), Lemma \ref{L2.2},
and $\tau \d\le \|y^\d-F(x_j)\|$ for $0\le j<n_\d$ we have
\begin{align*}
\|y^\d-F(x_j)+T e_j\| &\le \d +\|y-F(x_j)+ T e_j\| \le \d + C K_0\|e_j\|_s^\beta \|e_j\|_{-a}\\
&\le \frac{1}{\tau} \|y^\d-F(x_j)\| +CK_0\|e_0\|_s^\beta \|T e_j\|.
\end{align*}
This shows for $0\le j<n_\d$ that
\begin{align}
\|y^\d-F(x_j)\| &\le \frac{\tau}{\tau-1} \left(1+ CK_0\|e_0\|_s^\beta\right) \|T e_j\|,\label{151}\\
\|y^\d-F(x_j)\| &\ge \frac{\tau}{1+\tau} \left(1- CK_0\|e_0\|_s^\beta\right) \|T e_j\|.\label{152}
\end{align}
The inequalities (\ref{151}), (\ref{152}) and Lemma \ref{L2.3} imply that if $K_0\|e_0\|_s^\beta$ is sufficiently small then
\begin{equation}\label{212}
\|T e_j\| \lesssim  \|T e_{j+1}\|, \qquad 0\le j<n_\d-1.
\end{equation}
Consequently, we have from (\ref{151}) and (\ref{212}) that
\begin{equation}\label{3.18.1}
\|y^\d-F(x_j)\|\lesssim \|T e_{j+1}\|, \qquad 0\le j<n_\d-1.
\end{equation}
This together with (\ref{induction}) gives
\begin{equation}\label{213}
\|y^\d-F(x_j)\|\lesssim \|e_0\|_\mu s_{j+1}^{-\frac{a+\mu}{2(a+s)}}, \qquad 0\le j<l-1.
\end{equation}

Next we estimate $\|y-F(x_j)+T e_j\|$. We have from (\ref{5.1.2}) in Lemma \ref{L2.0},
Assumption \ref{A1} (a), and (\ref{induction}) that
\begin{align*}
\|y-F(x_j)-T e_j\|\lesssim K_0\|e_j\|_\mu^\beta \|e_j\|_{-a} \lesssim K_0\|e_0\|_\mu^\beta \|T e_j\|.
\end{align*}
Therefore, it follows from (\ref{212}) that
\begin{equation}\label{214}
\|y-F(x_j)-T e_j\|\lesssim K_0\|e_0\|_\mu^\beta \|T e_{j+1}\|, \qquad 0\le j\le l-1.
\end{equation}
On the other hand, by using (\ref{5.1.1}) in Lemma \ref{L2.0} and Assumption \ref{A1} (a), we have
\begin{align*}
\|y-F(x_j)+T e_j\| &\le K_0\|e_j\|_\mu^{\frac{a(1+\beta)-b}{a+\mu}} \|e_j\|_{-a}^{\frac{\mu(1+\beta)+b}{a+\mu}}\\
&\lesssim K_0\|e_j\|_\mu^{\frac{a(1+\beta)-b}{a+\mu}} \|T e_j\|^{\frac{\mu(1+\beta)+b}{a+\mu}}
\end{align*}
Therefore, it follows from (\ref{212}) and (\ref{induction})  that
\begin{align}\label{215}
\|y-F(x_j)-T e_j\| \lesssim K_0 \|e_0\|_\mu^{1+\beta} s_{j+1}^{-\frac{\mu(1+\beta)+b}{2(a+s)}},
\qquad 0\le j<l-1.
\end{align}

For the term $\|e_j\|$, we first have from the interpolation inequality (\ref{inter}), Lemma \ref{L2.2}, and
Assumption \ref{A1} (a) that
$$
\|e_j\|\le \|e_j\|_s^{\frac{a}{a+s}} \|e_j\|_{-a}^{\frac{s}{a+s}}
\lesssim \|e_0\|_s^{\frac{a}{a+s}} \|T e_j\|^{\frac{s}{a+s}}.
$$
With the help of (\ref{152}) we then obtain
\begin{equation}\label{216}
\|e_j\|\lesssim \|e_0\|_s^{\frac{a}{a+s}}\|y^\d-F(x_j)\|^{\frac{s}{a+s}}, \qquad 0\le j\le l-1.
\end{equation}
On the other hand, by using the interpolation inequality (\ref{inter}) and Assumption \ref{A1} (a)
we also obtain for $0\le j\le l-1$ that
\begin{align*}
\|e_j\|\le \|e_j\|_\mu^{\frac{a}{a+\mu}} \|e_j\|_{-a}^{\frac{\mu}{a+\mu}}
\lesssim \|e_j\|_\mu^{\frac{a}{a+\mu}} \|T e_j\|^{\frac{\mu}{a+\mu}}.
\end{align*}
This together with (\ref{212}) and (\ref{induction}) gives
\begin{equation}\label{217}
\|e_j\|\lesssim \|e_0\|_\mu s_{j+1}^{-\frac{\mu}{2(a+s)}}, \qquad 0\le j<l-1.
\end{equation}

Now we use (\ref{3.18.1}), (\ref{214}) and (\ref{216}) with $j=l-1$ and
use (\ref{213}), (\ref{215}) and (\ref{217}) for $0\le j<l-1$, we then obtain from (\ref{333}) that
\begin{align*}
\|T e_l\| &\le c_3 \|e_0\|_\mu s_l^{-\frac{a+\mu}{2(a+s)}} +\d
 + C K_0\|e_0\|_\mu^{1+\beta} \sum_{j=0}^{l-2} t_j (s_l-s_j)^{-1} s_{j+1}^{-\frac{\mu(1+\beta)+b}{2(a+s)}} \\
 &\quad \, + CK_0\|e_0\|_\mu^\beta \|T e_l\|
+CK_0\|e_0\|_s^{\frac{a\beta}{a+s}} t_{l-1}^{\frac{a-b}{2(a+s)}} \|y^\d-F(x_{l-1})\|^{\frac{s\beta}{a+s}} \|T e_l\| \\
&\quad\, + C K_0\|e_0\|_\mu^{1+\beta} \sum_{j=0}^{l-2} t_j (s_l-s_j)^{-\frac{b+a+2s}{2(a+s)}} s_{j+1}^{-\frac{\mu(1+\beta)+a}{2(a+s)}}.
\end{align*}
Since $\mu>s\ge (a-b)/\beta$, we can use Lemma \ref{L2} to derive that
\begin{align*}
\|T e_l\| &\le \left(c_3+C K_0\|e_0\|_\mu^\beta\right) \|e_0\|_\mu  s_l^{-\frac{a+\mu}{2(a+s)}}
+\d + CK_0\|e_0\|_\mu^\beta \|T e_l\| \\
&\quad \, +CK_0\|e_0\|_s^{\frac{a\beta}{a+s}} t_{l-1}^{\frac{a-b}{2(a+s)}}
\|y^\d-F(x_{l-1})\|^{\frac{s\beta}{a+s}} \|T e_l\|.
\end{align*}
Recall that (\ref{4.30.1}) in Lemma \ref{L2.2} implies $t_{l-1}\|y^\d-F(x_{l-1})\|^2\lesssim \|e_0\|_s^2$.
Since $s\ge (a-b)/\beta$ and $t_{l-1}\ge c_0>0$, we have
$$
t_{l-1}^{\frac{a-b}{2(a+s)}} \|y^\d-F(x_{l-1})\|^{\frac{s\beta}{a+s}}
\le \left(t_{l-1} \|y^\d-F(x_{l-1})\|^2\right)^{\frac{s\beta}{2(a+s)}} t_{l-1}^{\frac{a-b-s\beta}{2(a+s)}}
\lesssim \|e_0\|_s^{\frac{s\beta}{a+s}}.
$$
Therefore, noting $\|e_0\|_s\lesssim \|e_0\|_\mu$, we obtain
\begin{align}\label{218}
\|T e_l\| &\le \left(c_3+C K_0\|e_0\|_\mu^\beta\right) \|e_0\|_\mu  s_l^{-\frac{a+\mu}{2(a+s)}}
+\d + CK_0\|e_0\|_\mu^\beta \|T e_l\|.
\end{align}
Since $l<n_\d$, we have from the definition of $n_\d$ and (\ref{151}) that
\begin{equation}\label{219}
\d\le \frac{1}{\tau} \|y^\d-F(x_l)\|\le \frac{1}{\tau-1} \left(1+CK_0\|e_0\|_\mu^\beta\right) \|T e_l\|.
\end{equation}
Combining this with (\ref{218}) gives
$$
\|Te_l\|\le  \left(c_3+C K_0\|e_0\|_\mu^\beta\right) \|e_0\|_\mu s_l^{-\frac{a+\mu}{2(a+s)}}
+\left(\frac{1}{\tau-1}+ CK_0\|e_0\|_\mu^\beta \right)\|T e_l\|.
$$
Recall that $\tau>2$. Therefore, if $K_0\|e_0\|_\mu^\beta$ is sufficiently small, then we have
$$
\|Te_l\|\le \frac{2c_3(\tau-1)}{\tau-2}\|e_0\|_\mu s_l^{-\frac{a+\mu}{2(a+s)}}.
$$
Since $l\ge 1$ and $s_l\ge t_{l-1}\ge c_0$, we have $1+s_l\le (1+1/c_0) s_l$. Therefore
$\|T e_l\|\le C_* \|e_0\|_\mu (1+s_l)^{-\frac{a+\mu}{2(a+s)}}$
if we choose $C_*\ge 2c_3(1+1/c_0)(\tau-1)/(\tau-2)$.

Finally we will use (\ref{e29}) to show  the desired estimate for $\|(A^*A)^{\frac{s-\mu}{2(a+s)}} L^s e_l\|$.
Since we have verified the estimates for $\|T e_l\|$, the estimates (\ref{213}), (\ref{215}) and (\ref{217})
therefore can be improved to include $j=l-1$; this is clear from the above argument. Consequently we have from
(\ref{e29}) that
\begin{align*}
\|(A^*&A)^{\frac{s-\mu}{2(a+s)}} L^s e_l\| \\
&\le c_3 \|e_0\|_\mu + s_l^{\frac{a+\mu}{2(a+s)}} \d
+C K_0\|e_0\|_\mu^{1+\beta} \sum_{j=0}^{l-1}  t_j(s_l-s_j)^{-\frac{a+2s-\mu}{2(a+s)}} s_{j+1}^{-\frac{\mu(1+\beta)+b}{2(a+s)}}\\
& +C K_0\|e_0\|_\mu^{1+\beta} \sum_{j=0}^{l-1}  t_j (s_l-s_j)^{-\frac{b+2s-\mu}{2(a+s)}} s_{j+1}^{-\frac{\mu(1+\beta)+a}{2(a+s)}}.
\end{align*}
It then follows from Lemma \ref{L2} that
\begin{align*}
\|(A^*A)^{\frac{s-\mu}{2(a+s)}} L^s e_l\| \le \left(c_2+ CK_0\|e_0\|_\mu^\beta\right) \|e_0\|_\mu + s_l^{\frac{a+\mu}{2(a+s)}}\d.
\end{align*}
With the help of (\ref{219}) and the estimate on $\|Te_l\|$, we obtain
$$
\|(A^*A)^{\frac{s-\mu}{2(a+s)}} L^s e_l\| \le \left(c_3+CK_0\|e_0\|_\mu^\beta\right) \|e_0\|_\mu
+\frac{C_*}{\tau-1} (1+CK_0\|e_0\|_\mu^\beta) \|e_0\|_\mu.
$$
Since $\tau>2$, we thus obtain $\|(A^*A)^{\frac{s-\mu}{2(a+s)}} L^s e_l\|\le C_* \|e_0\|_\mu$
for any $C_*\ge 4c_3(\tau-1)/(\tau-2)$ if $K_0\|e_0\|_\mu^\beta$ is sufficiently small.
The proof is therefore complete. \hfill $\Box$
\end{proof}

Now we are ready to complete the proof of Theorem \ref{T2}, the main result in this paper.

\vskip 0.3cm

\noindent {\it Proof of Theorem \ref{T2}.}
Considering Lemma \ref{L2.2} and Remark \ref{R3.1}, it remains only to derive the order optimal convergence rates.
When $n_\d=0$, the proof is standard. So we may assume $n_\d>0$. From Lemma \ref{P1} it follows that
$\|e_{n_\d-1}\|_\mu\lesssim \|e_0\|_\mu$. By using Lemma \ref{L2.3} and the definition of $n_\d$
we have $\|y^\d-F(x_{n_\d-1})\|\lesssim \d$, which together with (\ref{152}) implies
that $\|e_{n_\d-1}\|_{-a} \lesssim \|T e_{n_\d-1}\|\lesssim \d$. Therefore,
from the interpolation inequality (\ref{inter}) it follows that
$$
\|e_{n_\d-1}\|_s\le \|e_{n_\d-1}\|_\mu^{\frac{a+s}{a+\mu}} \|e_{n_\d-1}\|_{-a}^{\frac{\mu-s}{a+\mu}}
\lesssim \|e_0\|_\mu^{\frac{a+s}{a+\mu}} \d^{\frac{\mu-s}{a+\mu}}.
$$
In view of (\ref{2.11}) in Lemma \ref{L2.2}, we consequently obtain
$\|e_{n_\d}\|_s\lesssim \|e_0\|_\mu^{\frac{a+s}{a+\mu}} \d^{\frac{\mu-s}{a+\mu}}$.
By using the definition of $n_\d$ and (\ref{1.2}) we have $\|y-F(x_{n_\d})\|\le (1+\tau) \d$.
Observing that (\ref{5.1.2}) in Lemma \ref{L2.0} and (\ref{2.11}) in Lemma \ref{L2.2} imply
\begin{align*}
\|T e_{n_\d}\| & \le \|y-F(x_{n_\d})\| +\|y-F(x_{n_\d}) +T e_{n_\d}\|\\
&\le \|y-F(x_{n_\d})\| +C K_0\|e_{n_\d}\|_s^\beta \|T e_{n_\d}\|\\
&\le \|y-F(x_{n_\d})\| +C K_0 \|e_0\|_s^\beta \|T e_{n_\d}\|.
\end{align*}
Thus, if $K_0\|e_0\|_s\lesssim K_0\|e_0\|_\mu$ is sufficiently small, then $\|T e_{n_\d}\|\lesssim \|y-F(x_{n_\d})\|$.
Consequently $\|e_{n_\d}\|_{-a}\lesssim \|T e_{n_\d}\|\lesssim \d$.
Now we can use again the interpolation inequality (\ref{inter}) to derive for all $r\in [-a,s]$ that
$$
\|e_{n_\d}\|_r\le \|e_{n_\d}\|_s^{\frac{a+r}{a+s}} \|e_{n_\d}\|_{-a}^{\frac{s-r}{a+s}}
\lesssim \|e_0\|_\mu^{\frac{a+r}{a+\mu}} \d^{\frac{\mu-r}{a+\mu}}.
$$
The proof is therefore complete. \hfill $\Box$

\section{\bf Conclusions}
\setcounter{equation}{0}

Inexact Newton regularization methods have been suggested by Hanke and Rieder in \cite{H97} and \cite{R99},
respectively, for solving nonlinear ill-posed inverse problems.
The convergence rates of these methods have been considered in \cite{R99,R01},
the results however turned out to be inferior to the so-called order optimal rates. For a long time it has been
an open problem whether these inexact Newton methods are order optimal, although the numerical illustrations
in \cite{R99,R01} present strong indication.

Important progress has been made recently in \cite{H2010} where the regularizing
Levenberg-Marquardt scheme is shown to be order optimal affirmatively. In this paper we considered a
general class of inexact Newton methods in which the inner schemes are defined by Landweber iteration,
the implicit iteration, the asymptotic regularization and Tikhonov regularization.
By establishing the monotonicity of iteration errors and deriving a series of subtle estimates,
we succeeded in proving the order optimality of these methods. We also extended these order
optimality results to a more general situation where the inner schemes are defined by linear
regularization methods in Hilbert scales. Our theoretical findings confirm the numerical results in \cite{R99,R01}.\\

\noindent
{\bf Acknowledgement.} Part of the work was carried out during the stay in Department of Mathematics at Virginia Tech.


\end{document}